\documentclass[10pt]{amsart}
\usepackage{hyperref,amssymb,amsmath,graphicx,verbatim,eufrak,amsthm, amsfonts, tabularx, caption, enumerate, comment, subfigure}
\usepackage{color, framed}
\input epsf.tex

\hypersetup{
    colorlinks=true,       
    linkcolor=red,          
    citecolor=green,        
    filecolor=magenta,      
    urlcolor=cyan           
}

\newtheorem{thm}{Theorem}[section]
\newtheorem{cor}[thm]{Corollary}
\newtheorem{lem}[thm]{Lemma}
\newtheorem{prop}[thm]{Proposition}

\newtheorem{ques}[thm]{Question}

\newtheorem*{clm*}{Claim}

\theoremstyle{definition}
\newtheorem{dfn}[thm]{Definition}
\theoremstyle{remark}

\newenvironment{lem*}[1]{\vspace{1ex}\noindent
{\bf Lemma* (#1).} [restatement]  \hspace{0.5em} \em }{ }
\newenvironment{thm*}[1]{\vspace{1ex}\noindent
{\bf Theorem* (#1).} [restatement]  \hspace{0.5em} \em }{ }

\newcommand{\set}[1]{\left\{#1\right\}}

\newcommand{\sr}[1]{\left(#1\right)}

\newcommand{\Integer}{\mathbb{Z}}

\newcommand{\Z}{\Integer}

\newcommand{\R}{\mathbb{R}}

\newcommand{\eps}{\varepsilon}

\newcommand{\eg}{{\em e.g.}~}

\renewcommand{\Pr}{}
\let\Pr\relax
\DeclareMathOperator{\Pr}{\mathbb{P}}

\def\squareforqed{\hbox{\rlap{$\sqcap$}$\sqcup$}}
\def\qed{\ifmmode\squareforqed\else{\unskip\nobreak\hfil
\penalty50\hskip1em\null\nobreak\hfil\squareforqed
\parfillskip=0pt\finalhyphendemerits=0\endgraf}\fi}

\newcommand{\ignore}[1]{ }

\newcommand{\dist}{\mathrm{dist}}

\usepackage[leftbars]{changebar}

\newcommand{\conn}{\leftrightarrow}

\usepackage{capt-of}

\newcommand{\perm}{\mathcal{P}}

\newcommand{\mbf}[1]{\mathbf{#1}}
\renewcommand{\L}{\mathcal{L}}

\begin{document}

\title{Tricolor percolation and random paths in 3D}

\author{Scott Sheffield}
\address{Scott Sheffield \\ Massachusetts Institute of Technology, Cambridge, MA}
\email{sheffield@math.mit.edu} 

\author{Ariel Yadin}
\address{Ariel Yadin \\ Ben Gurion University of the Negev, Beer Sheva, Israel}
\email{yadina@bgu.ac.il}  

\thanks{{\bf Acknowledgements.} This research was supported by Grant No.\  2010357
from the United States-Israel Binational Science Foundation (BSF) and by NSF grants DMS 064558 and 1209044.  We would like to thank Vladas Sidoravicius for useful discussions,
and also the European Science Foundation Short Visit Grant within the framework of the ESF Activity entitled 'Random Geometry of Large Interacting Systems and Statistical Physics'.}

\maketitle

\begin{abstract}
We study ``tricolor percolation'' on the regular tessellation of $\R^3$ by truncated octahedra, which is the three-dimensional analog of the hexagonal tiling of the plane.  We independently assign one of three colors to each cell according to a probability vector $p = (p_1, p_2, p_3)$ and define a ``tricolor edge'' to be an edge incident to one cell of each color.  The tricolor edges form disjoint loops and/or infinite paths.  These loops and paths have been studied in the physics literature, but little has been proved mathematically.

We show that each $p$ belongs to either the {\em compact phase} (in which the length of the tricolor loop passing through a fixed edge is a.s.\ finite, with exponentially decaying law) or the {\em extended phase} (in which the probability that an $n \times n \times n$ box intersects a tricolor path of diameter at least $n$ exceeds a positive constant, independent of $n$).  We show that both phases are non-empty and the extended phase is a closed subset of the probability simplex.

We also survey the physics literature and discuss open questions, including the following:  Does $p=(1/3,1/3,1/3)$ belong to the extended phase? Is there a.s.\ an infinite tricolor path for this $p$?  Are there infinitely many?  Do they scale to Brownian motion?  If $p$ lies on the boundary of the extended phase, do the long paths have a scaling limit analogous to SLE$_6$ in two dimensions?  What can be shown for the higher dimensional analogs of this problem?
\end{abstract}

\section{Introduction}
\subsection{Overview of the model} \label{subsec::overview}

Critical percolation on the faces of the hexagonal lattice has been very thoroughly studied, for example in celebrated works by Smirnov and by Smirnov and Werner
\cite{smirnov2001critical, smirnov2001critical2}.  As illustrated in Figure \ref{fig::hex}, if one colors each face one of two colors, the set of {\em bicolor edges} (i.e., edges that lie between two faces of distinct colors) forms a collection of paths and loops.  Smirnov's constructions can be used to show that as the mesh size tends to zero, the macroscopic loops converge in law to a random continuum collection of loops called a {\em conformal loop ensemble} \cite{camia2006two} (see also \cite{sheffield2009exploration, sun2011conformally}).  Each loop in the conformal loop ensemble looks locally like an instance of the {\em Schramm-Loewner evolution} with parameter $\kappa=6$ (written SLE$_6$) which is a particular random fractal non-self-crossing planar curve, first introduced by Schramm in 1999 \cite{schramm2000scaling}.  The existing theory of SLE curves relies heavily on conformal maps and the Riemann mapping theorem, and is very specific to two dimensions.

\begin{figure}[htbp]
\begin{center}
\includegraphics[width=2in]{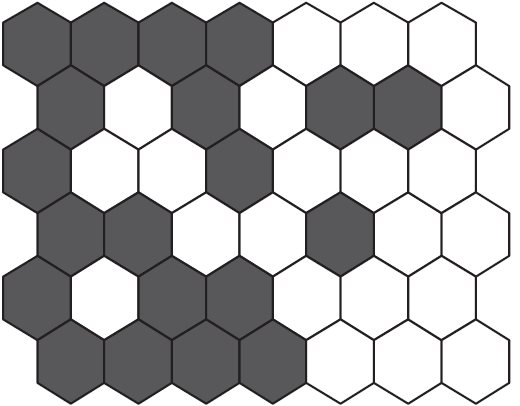}
\caption{ \label{fig::hex} The faces of the hexagonal lattice are each colored one of two colors; \textbf{\textbf{}}a {\em bicolor edge} is defined to be an edge incident to one face of each color.  The union of these edges is a collection of disjoint finite loops and/or infinite paths.}
\end{center}
\end{figure}

This paper will treat a three-dimensional analog of the percolation model mentioned above, in which the hexagon is replaced by the {\em truncated octahedron}, illustrated in Figure \ref{fig::onecell}.
Just as the hexagon tiles the plane, the truncated octahedron tessellates three-dimensional space, as illustrated in Figures \ref{fig::redyellowgrid} and \ref{fig::redyellowbluegreengrid}.\footnote{ This tessellation and the planar hexagonal tiling are the $d=2$ and $d=3$ cases of a more general tessellation of $\R^d$ by {\em permutohedra} (also spelled ``permutahedra'').  See \cite{ziegler1995lecture}, \cite{sloane1999sphere}, Wikipedia or appendix for more information.}  The simplest way to describe this tessellation is that it is the Voronoi tessellation corresponding to the set $$\mathcal L := 2 \mathbb Z^3 \cup \bigl(2 \mathbb Z^3 + (1,1,1) \bigr),$$ i.e., the set of vertices of $\mathbb Z^3$ whose coordinates are all even or all odd.  In particular, the cells in the tessellation are indexed by $\mathcal L$, and each $v \in \mathcal L$ is the center of the corresponding cell.  We will sometimes abuse notation by using $v$ to denote the cell itself.  Cells $v,w \in \mathcal L$ are adjacent along a square face if $v-w \in \{(\pm 2, 0, 0), (0, \pm 2, 0), (0, 0, \pm 2) \}$, and adjacent along a hexagonal face if $v-w \in \{(\pm 1, \pm 1, \pm 1) \}$, and otherwise non-adjacent.  The lattice $\mathcal L$ endowed with this adjacency relation is called the {\em body centered cubic lattice}.

\begin{figure}[htbp]
\begin{center}
\includegraphics[width=1in]{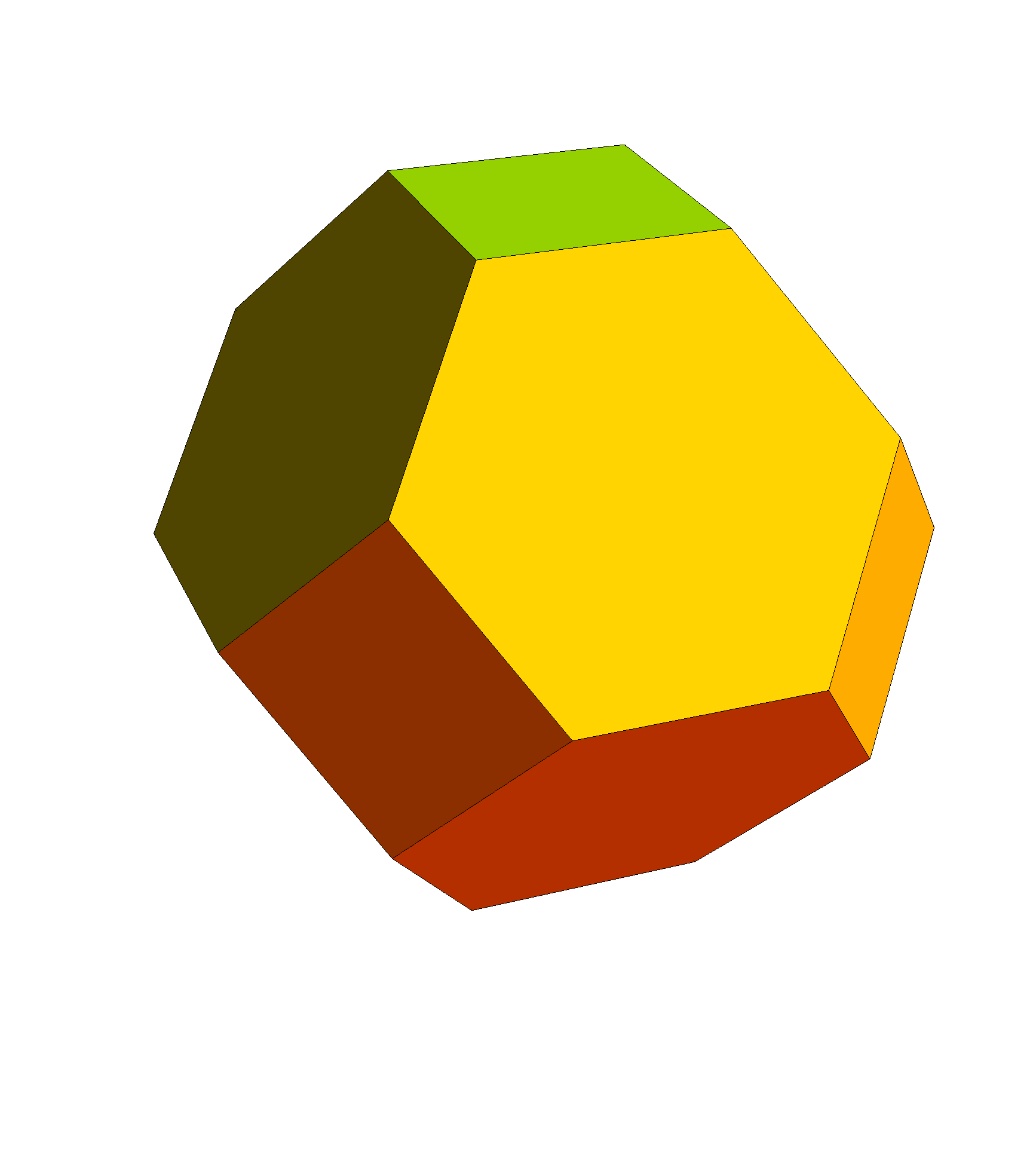}
\vspace{-.3in}\caption{ \label{fig::onecell} The truncated octahedron has six square and eight hexagonal faces.  It can be placed in $\mathbb R^3$ in such a way that the center is a point in $\mathcal L$ and the centers of the six square faces are at the neighboring vertices $\{ v \pm (1,0,0), v \pm (0,1,0), v  \pm (0,0,1) \}$, which belong to $\mathbb Z^3 \setminus \mathcal L$.}
\end{center}
\end{figure}

\begin{figure}[htbp]
\begin{center}
\includegraphics[width=2in]{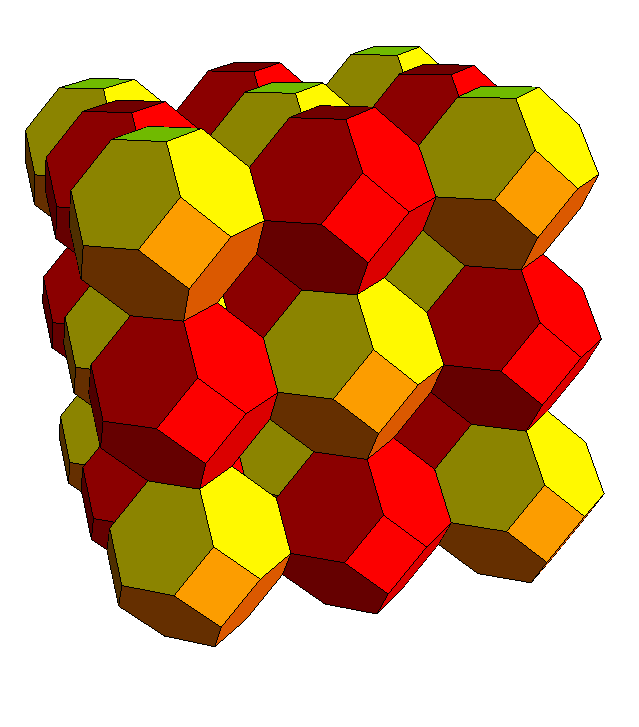}\vspace{-.2in} \caption{\label{fig::redyellowgrid} Shown are cells, as described in the caption of Figure \ref{fig::onecell}, centered at vertices $v \in 2\mathbb Z^3$.  They are colored red or yellow depending on whether the sum of the coordinates of $v$, modulo $4$, is equal to $0$ or $2$.}
\end{center}
\end{figure}

\begin{figure}[htbp]
\begin{center}
\includegraphics[width=2in]{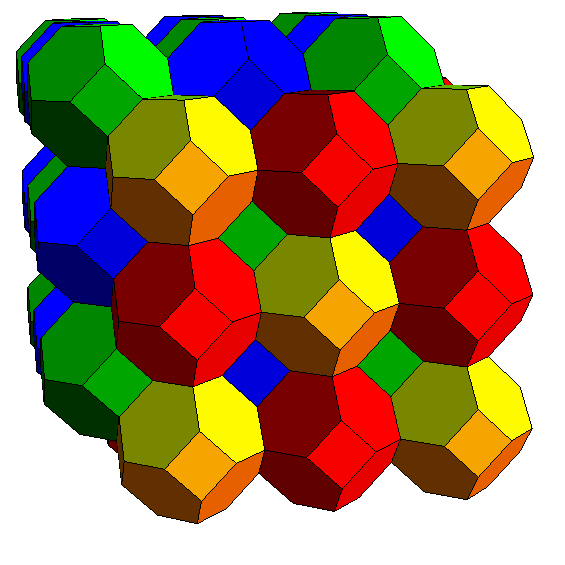}\vspace{-.2in}
\caption{\label{fig::redyellowbluegreengrid}Shown are the cells from Figure \ref{fig::redyellowgrid} together with cells centered at points in $v \in 2 \mathbb Z^3 + (1,1,1)$.  The latter are colored green or blue, depending on whether the sum of the coordinates of $v$, modulo $4$, is equal to $1$ or $3$.  This pattern extends to the full tessellation of $\mathbb R^3$.}
\end{center}
\end{figure}

The tessellation determines a complex of cells, faces, edges, and vertices.  We say that two of these objects are {\em incident} to each other if one is entirely contained in the boundary of the other.  The reader may observe by studying the figures that every face is incident to two cells, every edge is incident to three cells, and every vertex is incident to four cells.  (This is analogous to the planar hexagonal tiling in which each edge is incident to two hexagons, and each vertex is incident to three hexagons.)

There is a one-to-one correspondence between faces and adjacent pairs of cells in $\mathcal L$.  For example, $\bigl\{ (3,3,3), (3,3,5) \bigr\}$ describes a square face and $\bigl\{ (3,3,3), (4,4,4) \bigr\}$ a hexagonal face.  Similarly, there is a one-to-one correspondence between edges and triples of mutually adjacent cells ($3$-cliques). For example, $\bigl\{ (3,3,3), (3,3,5), (4,4,4) \bigr\}$ describes an edge (the edge incident to all three of those cells).  Finally, there is a one-to-one correspondence between vertices and quadruples of mutually adjacent cells ($4$-cliques).  For example, $\bigl\{ (3,3,3), (3,3,5), (4,4,4), (2,4,4) \bigr\}$ describes a vertex (the vertex incident to all four cells). When cells, faces, edges, and vertices are respectively represented by sets of $1$, $2$, $3$, and $4$ mutually adjacent elements of $\mathcal L$, the incidence relation corresponds to the subset-or-superset relation.

In this paper, we fix a vector $p = (p_1, p_2, p_3)$ belonging to the simplex
 $$\mathcal T:= \{ p \in [0,1]^3 : p_1+p_2+p_3 = 1\}$$ and then independently assign one of the three colors (red, yellow, and blue --- $i$th color chosen with probability $p_i$) to each of the cells in $\mathcal L$.  Let $\sigma: \mathcal L \to \{\mathrm{red} ,\mathrm{yellow},\mathrm{blue} \}$ denote the random color assignment.  An instance of such a coloring with $p=(1/3,1/3,/1,3)$ is illustrated in Figure \ref{fig::randomtricoloring}.  We then define a {\em tricolor edge} to be an edge that is incident to cells of all three colors, as illustrated in Figure \ref{fig::tricoloredge}.  Such an edge is represented by a triple of mutually adjacent cells in $\mathcal L$, each assigned a different color by $\sigma$.
A {\em tricolor vertex} is a vertex incident to cells of all three colors --- and represented by four cells in $\mathcal L$: two of one color, and one of each other color.  Since every tricolor vertex is incident to exactly two tricolor edges, and every tricolor edge is incident to exactly two tricolor vertices, the tricolor edges and vertices form loops and/or infinite paths (as the bicolor edges do in the planar hexagonal tiling).  We are interested in studying the existence and behavior of long tricolor paths.

A recent survey of this and similar models was given by Nahum and Chalker \cite{nahum2012universal}.  (See also \cite{nahum20113d}.)  According to \cite{nahum2012universal} the tricolor percolation model was first introduced by Scherrer and Frieman \cite{scherrer1986cosmic} as an enhancement of work of Vachaspati and Vilenkin \cite{vachaspati1984formation}.  It was studied via Monte Carlo simulations in a series of papers by Bradley, Debierre, and Strenski in 1992 \cite{bradley1992anomalous, bradley1992growing, bradley1992novel}, and has since been used, e.g., in \cite{hindmarsh1995statistical}.    The Nahum and Chalker paper also considers what should be involved in a continuum field theory associated to this model (keywords include $CP^{k|k}$, supersymmetry, replica limit), but it is not clear how to translate these ideas into mathematical conjectures.

\begin{figure}[htbp]
\begin{center}
\includegraphics[width=3in]{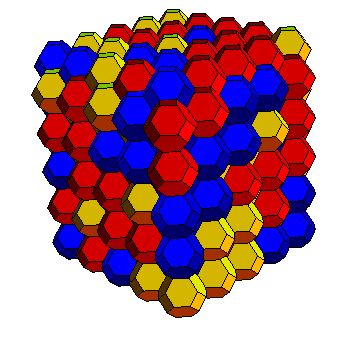}
\caption{\label{fig::randomtricoloring}Random $p=(1/3,1/3,1/3)$ tricoloring.}
\end{center}
\end{figure}

\begin{figure}[htbp]
\begin{center}
\includegraphics[width=2in]{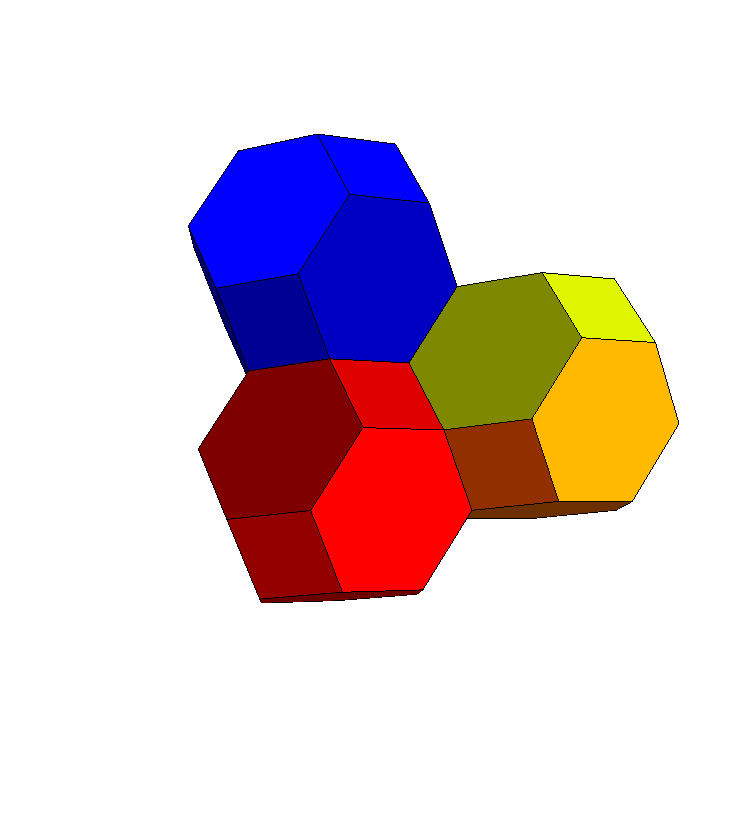} \includegraphics[width=2in]{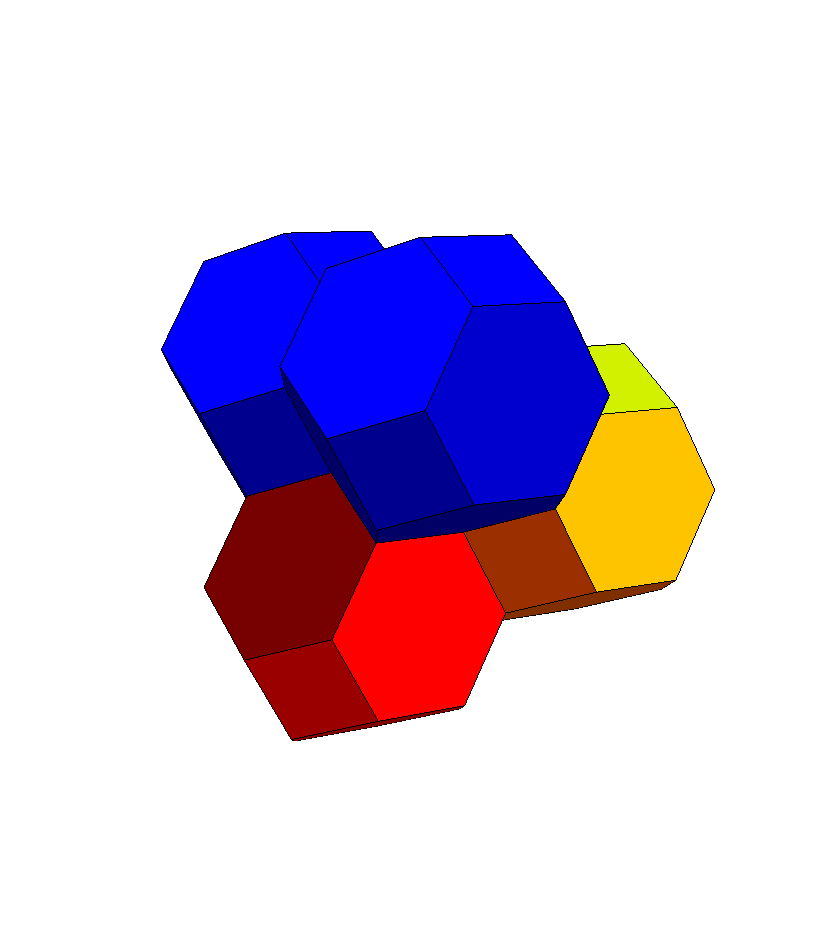}
\caption{\label{fig::tricoloredge} The left figure shows the three cells incident to a tricolor edge $e_1$.  The right figure shows these cells together with one additional blue cell, which is incident to one of the endpoints of $e$.  This blue cell, together with the original red and yellow cells, is incident to a second tricolor edge $e_2$, which has a (tricolor) endpoint vertex in common with $e_1$.  The set of all tricolor edges forms a collection of loops and/or infinite paths (like bicolor edges in two-color colorings of the hexagonal lattice).}
\end{center}
\end{figure}

The various simulations surveyed in \cite{nahum2012universal} suggest a phase transition: for certain values of $p$, including $(1/3,1/3,1/3)$, the origin cell has a positive probability of being incident to an edge of an infinite (as far as the simulation can detect) tricolor path, and this path appears to have (like Brownian motion) scaling dimension $2$.  These $p$ values are sometimes said to belong to the {\em extended phase} (though we will give a slightly different definition of the term ``extended phase'' below).  On the other hand, it is not hard to see that for some values of $p$, the length of a tricolor path starting at the origin will be a.s.\ finite, with a law that decays exponentially.  These $p$ are said to belong to the {\em compact phase}.  This is in particular the case if one of the $p_i$ lies below the threshold for site percolation on the body centered cubic lattice.  To see this, suppose $p_1$ is subcritical and note that the length of a tricolor path including an edge on the origin cell is bounded above by the number of edges incident to the largest cluster of red cells containing the origin (or an origin-adjacent cell); it is well known that the latter number has exponentially decaying law in the subcritical phase (see, {\em e.g.}, the reference text \cite{MR1707339}).  The critical probability for site percolation on the body centered cubic lattice has been estimated by Monte Carlo approximations as $p_c(\L) \approx .246$ \cite{1983JPhA...16..783G, 1991PhRvB..44...76B, 1998JPhA...31.8147L}, with the most recent estimate claiming significance to several decimal places: $p_c(\L) \approx .2459615(10)$ \cite{1998JPhA...31.8147L}.  The fact that $p_c(\L) < 1/3$ also follows rigorously from our arguments, see Corollary \ref{cor::explicitcriticalpercolationbound}.  The extended phase must be a proper subset of the triangle $\{p \in \mathcal T: p_1 \geq p_c(\L), p_2 \geq p_c(\L), p_3 \geq p_c(\L) \}$.

Although this has not been proved mathematically, it seems natural to guess that the extended phase is a convex subset of this triangle, centered at $(1/3,1/3,1/3)$, whose boundary is a simple curve (something like the brown region shown in Figure \ref{fig::phasediagram}) and that this curve divides the compact and extended phases from each other.   Numerical explorations have attempted to identify a $p$ on the boundary between these phases, and have found that for such a $p$, the long path seems to have a scaling dimension of about $5/2$ \cite{bradley1992anomalous, bradley1992growing,  bradley1992novel}.

\begin{figure}[htbp]
\begin{center}
\subfigure[$100$ steps]{\includegraphics[width=0.45\textwidth]{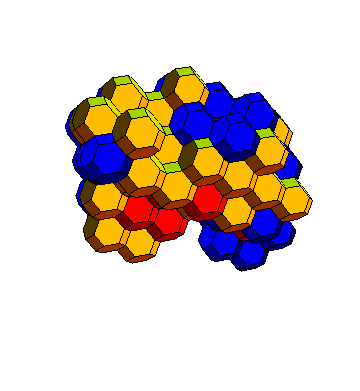}}
\hspace{0.02\textwidth}
\subfigure[$1000$ steps]{\includegraphics[width=0.45\textwidth]{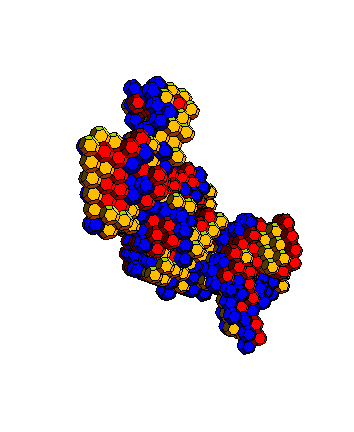}}
\subfigure[$10000$ steps]{\includegraphics[width=0.45\textwidth]{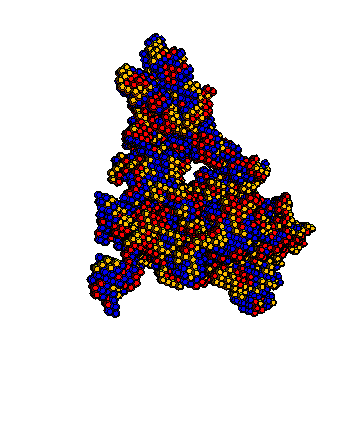}}
\hspace{0.02\textwidth}
\subfigure[\label{fig::chainssub}$100000$ steps]{\includegraphics[width=0.45\textwidth]{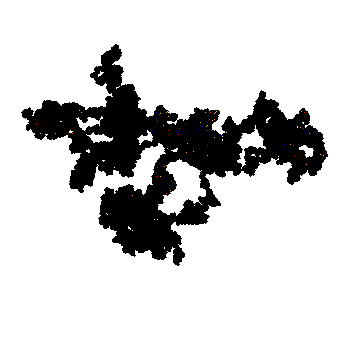}}
\end{center}
\caption{\label{fig::chains} \footnotesize{Random tricolor chains of length 100, 1000, 10000, and 100000, found in the $p=(1/3,1/3,1/3)$ model, that have not (yet) formed loops.}}
\end{figure}

To conclude our overview of the model, and to encourage further participation in this subject, we remark that the figures in this paper are remarkably easy to produce using math packages with built-in polyhedron functionality.  Such packages allow users to rotate the figures with a mouse and view them from different angles.  In Mathematica 8.0, for example, the following code defines a function ``PlaceCell'' (which puts a cell of color $D$ at a location $(A,B,C)\in \mathcal L$) and then uses it to generate the grid of cells shown in Figure \ref{fig::redyellowgrid}.

\begin{framed}
{\bf \small
\begin{verbatim}PlaceCell[{{A_, B_, C_}, D_}] =
  Translate[Scale[Rotate[{Switch[D, 1, Red, 2, Yellow, 3, Blue, 4, Green],
PolyhedronData["TruncatedOctahedron", "Faces"]},
  45 Degree, {0, 0, 1}], 1/(Sqrt[2])], {A, B, C}];
n = 3; Graphics3D[ Table[PlaceCell[{{2 i, 2 j, 2 k}, 1 + Mod[i + j + k, 2]}],
{i, 1, n}, {j, 1, n}, {k, 1, n}], Boxed -> False, Background -> White]
\end{verbatim}
 } \vspace{-.1in}
\end{framed}

Similarly, the following code uses the ``PlaceCell'' function to generate the chains in Figure \ref{fig::chains}.  It first randomly colors vertices of a (sufficiently large) $n \times n \times n$ subset of $\mathbb Z^3$ (which includes $\mathcal L$), storing them in the array ``ColorGrid''.  A quadruple of cells (``Vert'' and ``InitVert'') is used to describe a tricolor vertex, but we also use the ordering of the cells in this quadruple to encode a directed tricolor edge terminating at that vertex: the first three cells are the red, blue, and yellow cells incident to the edge, and last is the cell the edge points to.  The code includes a little piece of logic for replacing such a directed tricolor edge with the next directed tricolor edge along the tricolor path, and this logic is iterated a number of times (denoted ``Steps'') or until a loop is formed.  The code could presumably be made asymptotically more efficient by only assigning colors dynamically to the cells hit by the path (perhaps storing these cells and their colors in an efficient lookup table), instead of assigning them to an entire $n \times n \times n$ grid.

\begin{framed}
{\bf \small
\begin{verbatim}
n = 100; Steps = 500; Vert =
 InitVert = {{n/2, n/2, n/2 }, {n/2 + 1, n/2 + 1, n/2 + 1},  {n/2 - 1,
     n/2 + 1, n/2 + 1},  {n/2, n/2, n/2 + 2}}; ColorGrid =
 Table[Random[Integer, {1, 3}], {i, 1, n}, {j, 1, n}, {k, 1, n}];
ColorGrid[[n/2, n/2, n/2]] = 1; ColorGrid[[n/2 + 1, n/2 + 1, n/2 + 1]] = 2;
ColorGrid[[n/2 - 1, n/2 + 1, n/2 + 1]] = 3;
NewEnd[{A_, B_, C_, D_}]=D - 2 Sign[Mod[3 D - A - B - C, 3] (3 D - A - B - C)];
Chain = Table[{Vert[[ Mod[i - 1, 4] + 1]] , i}, {i, 1, Steps}];  j = 4;
loopmade = 0; While[j <= Steps && loopmade == 0, oldtarget = Vert[[4]];
Vert[[4]] =  Vert[[oldcolor =
    ColorGrid[[oldtarget[[1]], oldtarget[[2]], oldtarget[[3]]]]]];
Vert[[oldcolor]] = oldtarget; Vert[[4]] = NewEnd[Vert];
Chain[[j]] = {oldtarget, oldcolor}; If[Vert == InitVert, loopmade = 1]; ++j];
If[loopmade == 1, Print["Loop!"]];  Graphics3D[
 Table[PlaceCell[Chain[[i]]], {i, 1, j - 1}], Boxed -> False]
\end{verbatim}
 } \vspace{-.1in}
\end{framed}

\begin{figure}[htbp]
\begin{center}
\includegraphics[width=1.5in]{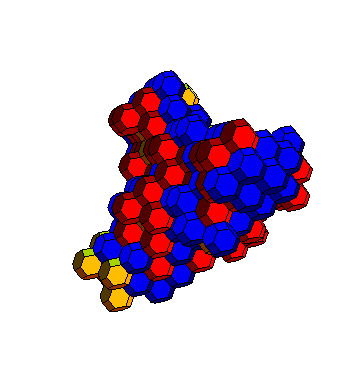}
\vspace{-.4in}
\caption{Cells incident to finite tricolor loop found in $p=(1/3,1/3,1/3)$ model.}
\end{center}
\end{figure}

\subsection{New results} \label{subsec::newresults}

Write $A_n$ for the {\em cubic $n$-annulus}, which we define to be the subset of $\mathcal L$ formed by starting with a radius $3n$ cube and removing a radius $n$ cube from its center:
$$A_n := ([-3n,3n]^3 \setminus [-n,n]^3) \cap \mathcal L.$$
We now present our first formal definition of the compact and extended phases:

\begin{dfn} \label{def::extendeddef}
Let $E_n$ be the event that there is a tricolor path from a vertex on the interior boundary of $A_n$ to a vertex on the exterior boundary of $A_n$.  A probability vector $p \in \mathcal T$ lies in the {\em compact phase} if $\Pr_p(E_n)$ tends to zero exponentially fast as a function of $n$.  It lies in the {\em extended phase} if $\Pr_p(E_n)$ is bounded below by a positive constant independently of $n$.
\end{dfn}

Here and in the sequel
$\Pr_p$ denotes the probability measure on colorings of $\L$ in which colors are assigned independently according to the probability vector $p \in \mathcal T$.  Note that one could alternatively use the term ``compact phase'' to mean the set of $p$ for which there a.s.\ exists no infinite tricolor path, or (another alternative) the set of $p$ for which the expected length of the tricolor path through the origin is finite.  As we discuss in Section \ref{subsec::open}, it is reasonable to conjecture that these definitions describe the same set, except possibly along a critical phase separation curve.  This is analogous to the situation in classical percolation theory, where one can define the critical percolation threshold in various ways, and it takes some work to prove that different definitions are equivalent.  We will use Definition \ref{def::extendeddef} in this paper in part because it sets up a natural dichotomy that is relevant to the statements we are able to prove.  To begin with, we are able to prove that each $p$ belongs to exactly one of these two phases:

\begin{figure}[htbp]
\begin{center}
\includegraphics[width=4in]{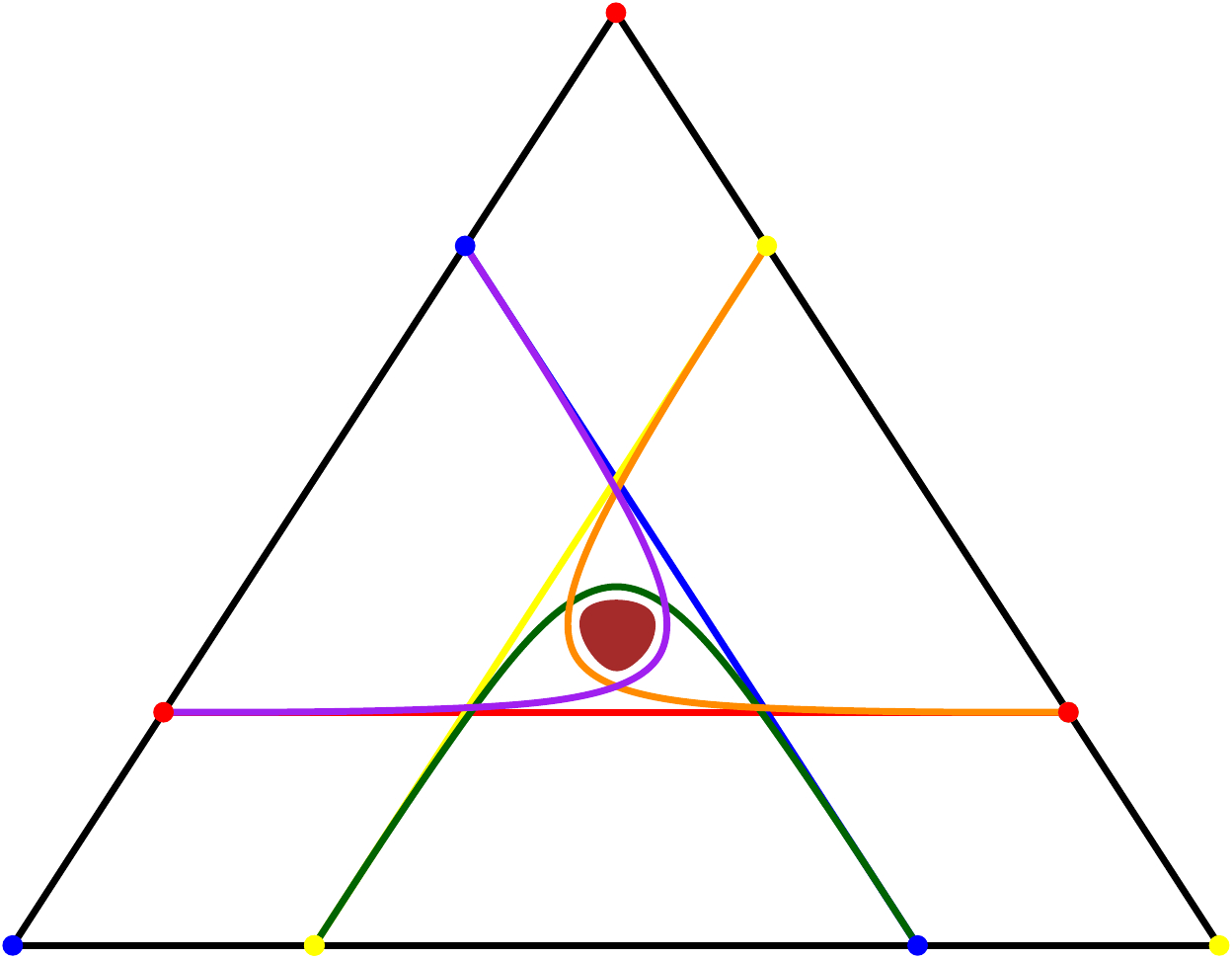}
\caption{\label{fig::phasediagram} This diagram represents the triangle of possible choices for $p=(p_1, p_2, p_3)$, with the three vertices (colored red/yellow/blue) representing
respectively the all red/yellow/blue extreme points: $(1,0,0)$, $(0,1,0)$ and $(0,0,1)$.  The red/yellow/blue lines represent, respectively, the thresholds for the existence of infinite red/yellow/blue clusters of cells (i.e., the lines $p_i = p_c(\L)$ where numerics suggest $p_c(\L) \approx .2459615(10)$ \cite{1998JPhA...31.8147L}).  The orange/purple/green lines represent the threshold for existence of ``large'' clusters of the faces that lie between red-yellow/red-blue/blue-yellow cell pairs.  (More precisely, each curve, together with the line segment between its endpoints, bounds the region in which the probability that the corresponding face cluster containing a given face has size greater than $N$ does {\em not} decay exponentially with $N$; we expect but do not prove that this cluster has a positive probability of being infinite when $p$ is in the interior of this region.)  The middle brown region represents the {\em extended phase} which is (roughly speaking) the region in which there are macroscopic tricolor paths at all scales.  We expect (but do not prove) that when $p$ lies in the interior of this region there are a.s.\ infinite tricolor paths.  Simulations in \cite{bradley1992anomalous, bradley1992growing,  bradley1992novel} suggest a larger extended phase than the one sketched here: they find that the boundary of the extended phase intersects the line segment $p_2 = p_3$ (the vertical bisector of the probability triangle in the figure) at $p_1 \approx .255$ (just above the $p_1 \approx .246$ threshold for red percolation) and at $p_1 \approx .417$.
}
\end{center}
\end{figure}

\begin{thm} \label{thm::annulusphasecharacterization}
There exists a constant $\alpha \in (0,1)$ such that if $\Pr_p(E_n) < \alpha$, for some $n>1$, then $\Pr_p(E_n)$ tends to zero exponentially fast as a function of $n$.  Hence every $p \in \mathcal T$ belongs to either the compact phase or the extended phase.
\end{thm}

As we will see, this can be proved with a type of finite-range-dependent-percolation Peierls argument.  In fact, this argument also implies the following:
\begin{thm} \label{thm::exponentialdecayofpathlength}
If $p$ belongs to the compact phase, then the probability that an edge at the origin belongs to a tricolor path of length $L$ decays exponentially fast
as a function of $L$.
\end{thm}

This theorem suggests a stark dichotomy between the two phases.  In the compact phase, the tricolor loops are ``microscopic'', much like the clusters in subcritical percolation.  The length of a tricolor loop starting at the origin not only is a.s.\ finite, but also has exponentially decaying law.  In the extended phase, one has a positive lower bound on the probability of seeing a tricolor loop crossing $A_n$, independently of $n$.  Informally, this implies that one encounters macroscopic loops ``at any scale''.  Using what is probably the trickiest and least direct argument in the paper, we establish the following:

\begin{thm} \label{thm::closedandnonempty} The extended phased is a closed and non-empty subset of $\mathcal T$.\end{thm}

The tricky part is showing that there exists at least one $p$ in the extended phase.  A peculiar feature of the argument is that it does not allow us to prove that any {\em particular} $p$ belongs to the extended phase.  It does not even tell us whether $p = (1/3,1/3,1/3)$ belongs to the extended phase.

Using the primary color rule for pigments (imagine each cell is coated in wet paint of its given color), we make the following definitions:

\begin{enumerate}
\item An {\em orange face} is a face incident to one red and one yellow cell.
\item A {\em green face} is a face incident to one yellow and one blue cell.
\item A {\em purple face} is a face incident to one blue and one red cell.
\end{enumerate}

Two faces are said to be adjacent if they have a boundary edge in common.  It is natural to consider percolation on the graph of faces. 
What kinds of phase transitions does one have for the large green clusters?  For our purposes, it will turn out to be useful to consider when the size of the green cluster containing a given face is a.s.\ finite {\em and} has an exponentially decaying law.   Clearly, if either the cluster of blue cells incident to a given face or the cluster of yellow cells incident to a given face has exponentially decaying law, then the cluster of green faces must have exponentially decaying law also.  It is well known that the the size of the origin-containing cluster decays exponentially in subcritical percolation (see, e.g., the reference text \cite{MR1707339}) which implies that the size of the green cluster decays exponentially whenever either the blue or yellow probabilities are subcritical.  We will show more than this:

\begin{thm}  \label{thm::orangegreenpurple} The green phase separation curve in Figure \ref{fig::phasediagram} is the graph of a Lipschitz function with Lipschitz norm at most $\sqrt{3}$ (i.e., at most the slope of the upper two boundary lines of the triangle).  This graph has no points in common with the critical percolation lines, except on the boundary of $\mathcal T$.  In other words, the green curve in Figure \ref{fig::phasediagram} lies strictly below both the blue and yellow lines in Figure \ref{fig::phasediagram}.
\end{thm}

The first sentence of Theorem \ref{thm::orangegreenpurple}  will turn out to be a simple monotonicity observation, which follows from the fact that the probability that there exists a green cluster of size greater than $N$ is an increasing function of the pair $(p_2, p_3)$ (the blue and yellow probabilities).  The second statement requires a short argument, which will in fact give an explicit (but non-optimal) upper bound on the height of the green curve.
Since the existence of a long tricolor path through a vertex implies the existence of comparably large (up to constant factor) clusters of green, orange, and purple faces through that vertex, the extended phase (brown region in Figure \ref{fig::phasediagram}) is necessarily a subset of the region bounded between the green, orange and purple curves in Figure \ref{fig::phasediagram}.
This fact and Theorem \ref{thm::orangegreenpurple} together imply the following:

\begin{cor}
\label{cor:extended separated}
The boundary of the extended phase is of positive distance from each of the critical lines $p_1 = p_c$, $p_2 = p_c$, and $p_3 = p_c$ (where $p_c = p_c^{\mathrm{site}}(\mathcal L)$).  In other words, the (closed) brown set in Figure \ref{fig::phasediagram} is bounded away from the red, yellow, and blue lines in Figure \ref{fig::phasediagram}.
\end{cor}

Note that Corollary \ref{cor:extended separated} in particular implies that there exists a $p = (p_1, p_2, p_3)$ for which there exist infinite clusters of all three colors, but nonetheless $p$ lies in the compact phase (and hence the length of the tricolor path through a given vertex has exponentially decaying law).

\subsection{Open problems} \label{subsec::open}

As this section illustrates, our list of fundamental questions about tricolor paths is much longer than our list of fundamental results.  There are several embarrassingly simple questions that have not been settled mathematically.  To describe one of the most fundamental issues, note that one could consider the following subdivisions of what we call the extended phase:

\begin{enumerate}
\item {\bf No-path extended phase:}  $\Pr_p(E_n)$ is bounded below independently of $n$ (i.e., $p$ is in the extended phase), but there is a.s.\ no infinite tricolor path.
\item {\bf Single-path extended phase:} There is a.s.\ exactly one infinite tricolor path.
\item {\bf Many-path extended phase:} There are a.s.\ infinitely many infinite tricolor paths.
\end{enumerate}

However, we are not able to determine whether any one of these phases is empty or not.  One natural guess would be that the extended phase corresponds to a convex shape bounded by a simple boundary curve, such as the brown region in Figure \ref{fig::phasediagram}, and that the many-path extended phase corresponds to interior of that region, while the no-path extended phase corresponds to the boundary.  On the other hand, we have not yet even answered the following:

\begin{ques} \label{ques::zerooneinfinity} Is it true that, for any $p$, the number of infinite tricolor paths is a.s.\ $0$, $1$, or $\infty$? \end{ques}

Note that the arguments used to show that the number of percolation clusters a.s.\ belongs to $\{0,1\}$, as in \cite{burton1989density}, do not work here, because if there are two infinite tricolor paths, it is not possible to ``join'' them to each other by changing the colors of finitely many cells.  Assuming that the answer to Question \ref{ques::zerooneinfinity} is nonetheless yes, one can ask the following:

\begin{ques} Which of the extended sub-phases mentioned above (no-path, single-path, and many-path) are non-empty? \end{ques}
\begin{ques} Which one (if any) does $p = (1/3,1/3,1/3)$ belong to? \end{ques}
\begin{ques} Are there a.s.\ infinitely many infinite paths for all $p$ in the interior of extended phase?\end{ques}
\begin{ques} Is there a.s.\ an infinite path when $p$ is on the boundary of the extended phase? (This is analogous to the question of whether one has percolation at $p_c$.)  \end{ques}
\begin{ques} What can one say about the extended phase (or the sub-phases mentioned above) as a set?  Is it connected?  Is it convex?
\end{ques}

Let $\mathcal A$ be the set of $p$ values for which one has percolation of all three bi-color face types (green, orange and purple).  That is, $\mathcal A$ is the set bounded between the green, orange, and purple curves in Figure \ref{fig::phasediagram}.  Then we ask the following:

\begin{ques} Is $\mathcal A$ connected?  Is it convex?  Does the boundary of the extended phase intersect $\partial \mathcal A$?
\end{ques}

The following questions address the existence of scaling limits.  One would expect any scaling limit to be rotationally invariant, but we cannot prove that this is necessary.

\begin{ques} Do the infinite paths scale to Brownian motion when $p$ is in the interior of the extended phase?\end{ques}
\begin{ques} What happens if $p$ is on the boundary of the extended phase?  Is there a different kind of scaling limit (perhaps a higher dimensional analog of SLE$_6$) in this case?  Does the scaling limit depend on which $p$ on the boundary of the extended phase is chosen?  (Some physicists have speculated that different $p$ on the boundary of the extended phase should correspond to the same field theory \cite{nahum2012universal}, so it is reasonable to speculate that they might also correspond to the same type of random path.)
\end{ques}

As mentioned earlier in a footnote, the tiling of $\R^2$ by hexagons and the tessellation of $\R^3$ by truncated octahedra are both special cases of the so-called {\em permutohedron} tessellation of $\R^d$, as we discuss in the appendix.  This tessellation is classical and has appeared in many papers and contexts, e.g. \cite{kitto1994isomorphism, sloane1999sphere, ziegler1995lecture, 2013arXiv1301.3400T}. An important aspect of this tessellation in $d$ dimensions is that one has $d$ cells sharing each edge and $d+1$ cells sharing each vertex, and thus one can randomly assign each cell one of $d$ colors and consider the paths comprised of {\em $d$-color edges} (or ``full-spectrum edges''), which we define to be edges incident to one cell of each color.  There are other tessellations with this property, but this one is particularly simple and canonical.  One might expect that when $d$ was large it would be easier to show that $d$-color paths have Brownian motion as a scaling limit (perhaps using lace expansions or related techniques).

\begin{ques} Can convergence to Brownian motion be established in sufficiently high dimension $d$ when $p_1 = p_2 = \ldots = p_d = 1/d$? \end{ques}
\begin{ques} Do the long tricolor paths have a scaling limit (and if so what kind) when $d$ is large and $p$ is on the boundary of the extended phase? \end{ques}

We remark that, although we will not do this here, we believe that the proofs of the main results of this paper (the results described in Section \ref{subsec::newresults}) could in principle be extended to any dimension $d \geq 3$.

\subsection{Vortex line interpretation}
In this section we briefly remark that there is a standard ``vortex line'' interpretation of the tricolor path model, in which each tricolor edge comes with an associated unit of ``flow'', with a direction determined by the cyclic red-yellow-blue ordering, and no flow is assigned to other edges.

To construct this flow in slightly different way, recall that we have defined $\sigma$ as a random function on $\mathcal L$.
 Using this, we obtain a function $\sigma'$ on ordered pairs $(v,w)$ of adjacent elements in $\mathcal L$ by $\sigma'(v,w) = \eta(\sigma(v), \sigma(w))$ where
$$\eta( \textrm{red}, \textrm{yellow}) = \eta( \textrm{yellow}, \textrm{blue})
= \eta( \textrm{blue}, \textrm{red}) = -1,$$
$$\eta( \textrm{red}, \textrm{blue}) = \eta( \textrm{blue}, \textrm{yellow})
= \eta( \textrm{yellow}, \textrm{red}) = 1,$$
$$\eta( \textrm{red}, \textrm{red}) = \eta( \textrm{yellow}, \textrm{yellow})
= \eta( \textrm{blue}, \textrm{blue}) = 0.$$
If we consider a triangle with vertices labeled by the three colors (as in Figure \ref{fig::phasediagram}), then $\sigma$ is a map from $\mathcal L$ to the vertices of the triangle, and $\sigma'$ describes whether this function goes counterclockwise, goes clockwise, or stays constant as one moves from $v$ to $w$.  Alternatively, we may interpret $\sigma$ a function to integers modulo $3$, and $\sigma'$ as a discrete gradient of $\sigma$.
If $(v,w,x)$ is a triple of mutually adjacent vertices in $\mathcal L$, which describes a directed edge of the tessellation, then we can write $\sigma''(v,w,x) = \frac13\bigl(\sigma'(v,w) + \sigma'(w,x) + \sigma'(x,v)\bigr)$.  This quantity (which can be interpreted as a ``discrete curl'' of $\sigma'$) is zero unless the edge is tricolor, in which case it is $1$ or $-1$, depending on the orientation of the edge.  In a sense, a directed tricolor path (comprised of a sequence of directed edges on which $\sigma''$ is equal to $1$) is a ``vortex line'' of $\sigma'$, and the existence of tricolor edges corresponds to the failure of the function $\sigma'$ to be a discrete gradient of an integer-valued function on $\mathcal L$.  If one has a non-self-intersecting cyclic loop $v_0,v_1, v_2, \ldots, v_k = v_0$ in $\mathcal L$, then one can interpret $\sigma'(v_0, v_1) + \sigma'(v_1,v_2) + \ldots + \sigma'(v_{k-1}, v_k)$ as ($3$ times) the total amount of $\sigma''$ flow passing through a surface bounded by this loop.  The following is easy to prove (e.g., by gradually retracting $\partial S$ to a point in such a way that it passes through one edge at a time).

\begin{prop} \label{prop::flowthroughsurface} Let $S$ be a smooth surface with smooth boundary, homeomorphic to a closed disc, embedded in $\R^3$ in such a way that it does not intersect any vertices of the truncated octahedron tessellation.  Let $v_0, v_1, \ldots v_k$ be the ordered sequence of cells encountered by the loop $\partial S$.  Then the number of tricolor edges passing through the surface $S$ (counted with sign) is
$$\frac13 \bigl( \sigma'(v_0, v_1) + \sigma'(v_1,v_2) + \ldots + \sigma'(v_{k-1}, v_k) \bigr).$$
\end{prop}

Figure \ref{fig::straightpath} is meant to provide some intuition about the flow interpretation.  Once we condition on the cells along a long tricolor path, there is an ``expected flow'' in the opposite direction.  When we continue the long tricolor path, this phenomenon in some sense ``encourages'' the path to retrace its past, instead of exploring new territory.  If one of the $p_i$ is very close to $1$ (which in particular implies that $p$ is in the compact phase) then this effect will be overwhelming (and the continuation of the tricolor path segment between the cells shown in the figure will indeed stay close to the path segment with high probability, until it forms a loop).

\begin{figure}[htbp]
\begin{center}
\includegraphics[width=2in]{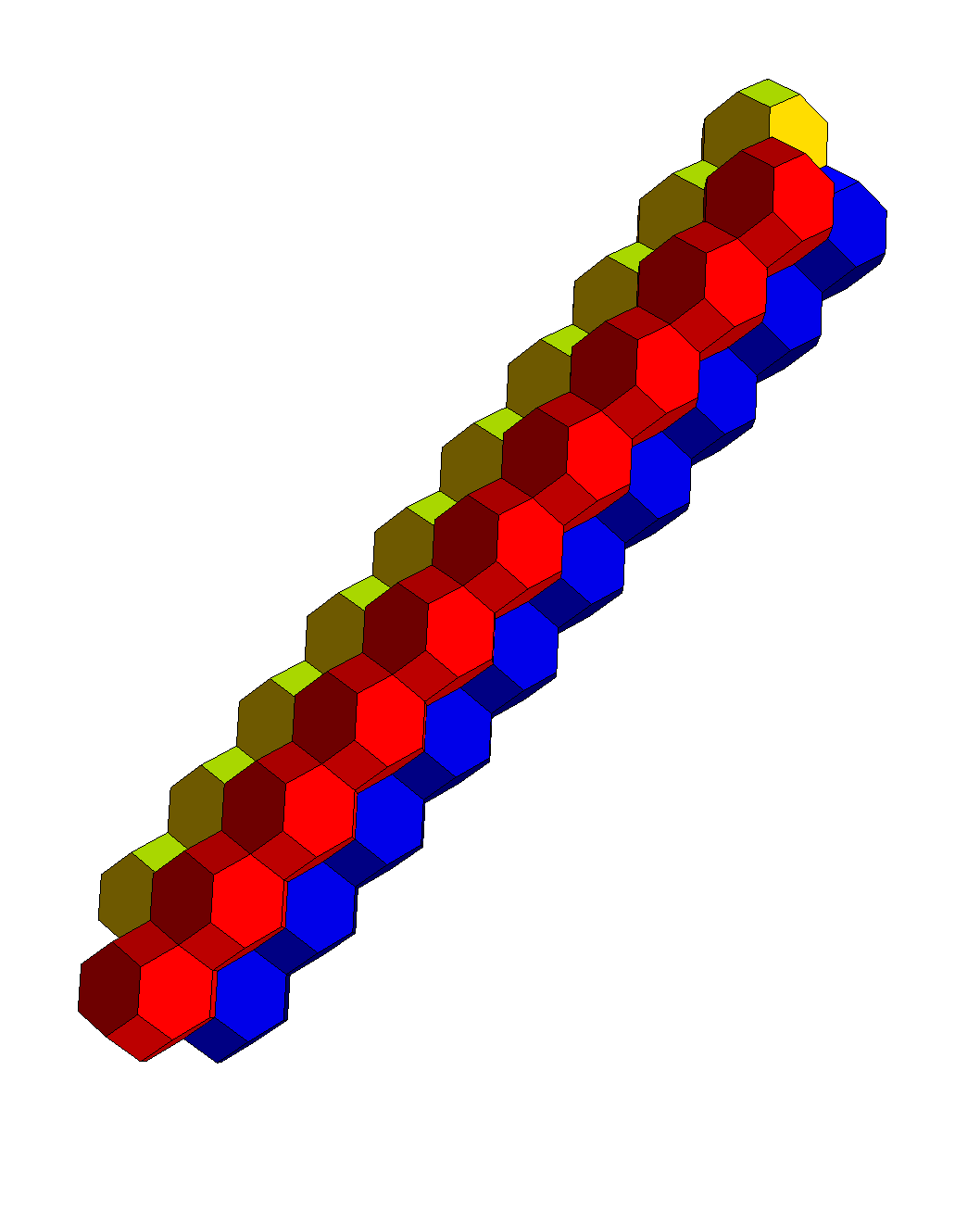}
\includegraphics[width=2.4in]{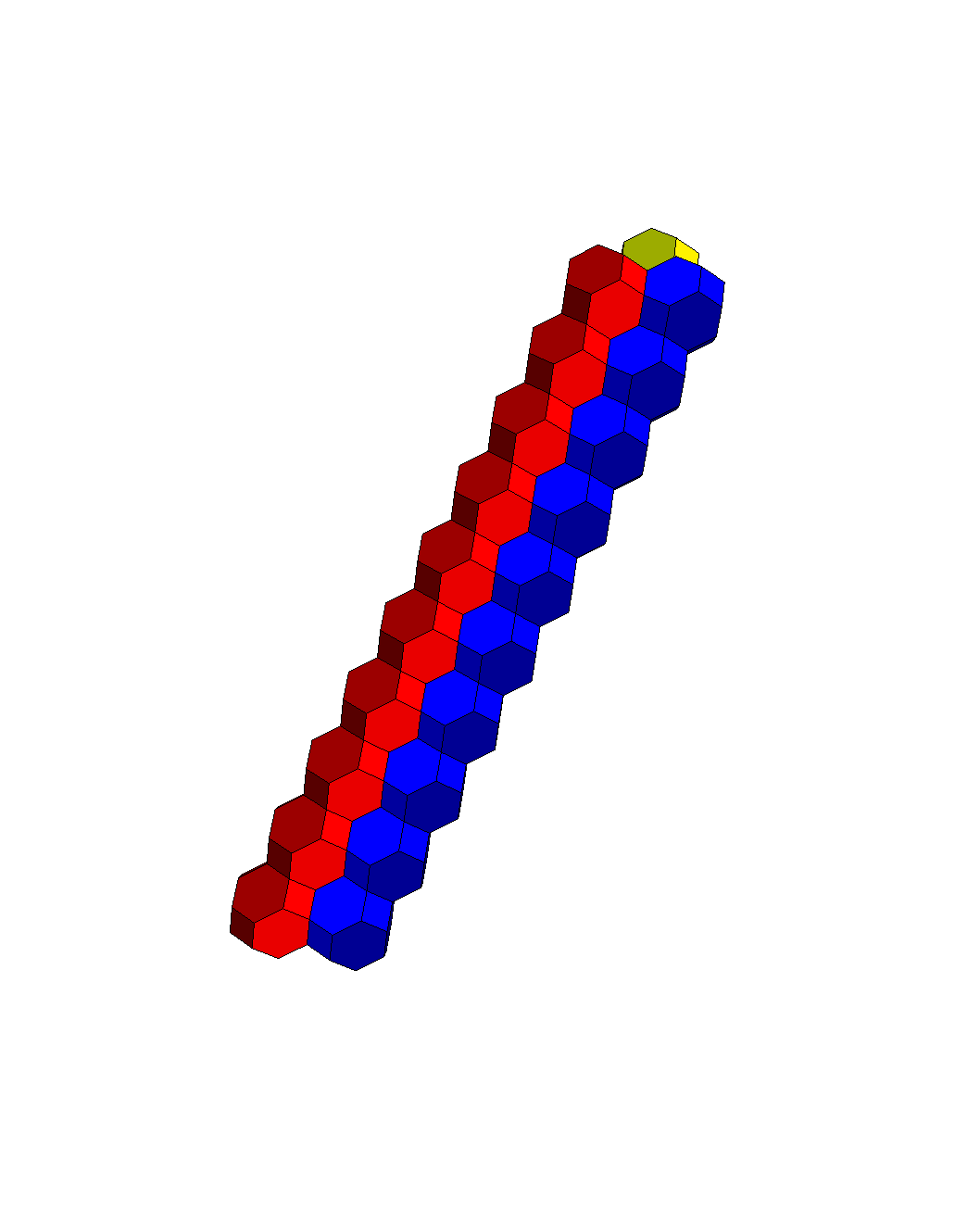}
\caption{\label{fig::straightpath}Two rotated views of a path obtained by letting $i$ range from $1$ to $10$ and coloring cells $(i,i,i)$ red, cells $(i,i,i+2)$ yellow, and cells $(i-1,i+1,i+1)$ blue.  The result is a ``straight'' tri-color path in the $(1,1,1)$ direction.  There are three bicolor paths on the boundary of the cluster of cells, running from one endpoint to the other.  If we start with this configuration, and color the remaining cells according to $p=(p_1,p_2,p_3)$, then there is an expected flow of $p_1$ along the blue-yellow path, $p_2$ on the blue-red path, and $p_3$ along the yellow-blue path.  This expected flow runs in the opposite direction of the flow on the tricolor path.  In a sense, once we are given the colors of the cells on the long tricolor path, the flow along the path is offset by an ``expected flow'' in the opposite direction.}
\end{center}
\end{figure}

%
%

\section{Proofs}

In this section we prove our main results: Theorems \ref{thm::annulusphasecharacterization}, \ref{thm::exponentialdecayofpathlength}, \ref{thm::closedandnonempty}, and \ref{thm::orangegreenpurple}

\subsection{Proof of Theorems \ref{thm::annulusphasecharacterization} and
\ref{thm::exponentialdecayofpathlength} }

The following is a fairly standard observation about dependent percolation.  If a percolation model has only short range dependence, and each site separately has a small probability of being open, then the size of the origin-containing cluster has a law that decays exponentially.

\begin{prop}
\label{prop:Peierls}
For any $d,D$ there exist constants $\alpha = \alpha(d,D)<1$ and
$c_1 = c_1(d,D), c_2=c_2(d,D) > 0$
such that the following holds.
Consider a random site percolation $\sigma:\Z^d \to \{0,1\}$ with the property that for each $v \in \Z^d$, the value of $\sigma(v)$ is independent of the restriction of $\sigma$ to $\{w : \dist_{\Z^d}(v,w) > D \}$, and $\Pr [\sigma(v) = 1 ] \leq \alpha$.  Then
the probability that the origin is in an open
component of size at least $R$ is at most $c_1 \exp \sr{ - c_2 R }$.
In particular, the probability that $0$ is
connected to distance $R$ is at most $c_1 \exp \sr{ - c_2 R }$.
\end{prop}

\begin{proof}
Let $S$ be a finite connected subset of $\Z^d$ containing $0$.
By repeatedly removing cubes of radius $D$,
one can show that there exists a subset $B \subset S$ such that
\begin{itemize}
\item $|B| \geq |S| \cdot (2D+1)^{-d}$.
\item For every two vertices $a \neq b \in B$ we have that $\dist_{\Z^d}(a,b) > D$.
\end{itemize}
Thus, since vertices at distance greater than $D$ are independent,
$$ \Pr [ \textrm{ all vertices in $S$ are open } ] \leq \alpha^{|B|} \leq \alpha^{|S| \cdot (2D+1)^{-d} } . $$
It is well known that the number of possible choices for a finite connected subset $S$ containing $0$ of size $|S|=n$
is at most $C^n$ for some large enough constant $C=C(d)$ (in fact, $C=7^d$ suffices, see \eg \cite[Chapter 4.2]{MR1707339}).
Thus, if $\alpha < C^{-(2D+1)^d}$, 
then for any $R>0$,
the probability that $0$ is in an open
component of size at least $R$ is at most the probability that there
exists an open connected subset containing $0$ of size at least $R$, which is bounded by
$$ \sum_{k \geq R} \alpha^{k (2D+1)^{-d} } C^k \leq  \sr{ C \alpha^{(2D+1)^{-d} } }^R  \cdot
\frac{1}{1 - C \alpha^{(2D+1)^{-d} } } . $$
\end{proof}

It will be convenient to consider colorings of $\L$ with varying probabilities,
although still independent.
Given a function $f:\L \to \mathcal T$ we may color
the cells in $\L$ independently so that $\Pr [ \sigma(x) = j ] = f(x)_j$
for all $x \in \L , j \in \set{1,2,3}$ (for simplicity we have identified the colors red, yellow, blue with
$1,2,3$ respectively).
We denote this probability measure $\Pr_f$.

Given $A,B \subset \R^3$ we use the notation $A \conn B$ to denote the event
that some point in $A$ is connected by a tricolor path to some point in $B$.
We write $x \conn A$ for $\set{x} \conn A$.

The following lemma is a generalization of Theorem \ref{thm::annulusphasecharacterization}.
The proof of  Theorem \ref{thm::annulusphasecharacterization} follows by taking
$f \equiv p \in \mathcal{T}$ in the lemma.

\begin{lem}
\label{lem:compact phase}
There exist constants $\alpha<1$ and $c_1, c_2 >0$
such that the following holds.
Let $C_r(z) = \set{ x \in \R^3 \ : \ ||x-z||_\infty \leq r }$.
If there exists $r>0$ such that
$\sup_{x \in \L} \Pr_{f} [ C_r(x) \conn (C_{3r}(x))^c ] \leq \alpha$, then
for any $x \in \L$ and $R>0$,
$$ \Pr_{f} [ C_R(x) \conn (C_{3R}(x))^c ] \leq c_1 \exp \sr{ - c_2 R/r } . $$
\end{lem}

\begin{proof}
Fix $r > 2$.
Consider the following tessellation by $r$-cubes of $\R^3$:
Let $G_r = \{ C_r(z) \ : \ z \in 2r \cdot \Z^3 \}$.
Equip $G_r$ with a graph structure by letting $C_r(2rz) \sim C_r(2r z')$ if $z \sim z'$ in $\Z^3$.
This is just the Voronoi tessellation of $2r \Z^3$ in $\R^3$.
The graph $G_r$ is of course isomorphic to $\Z^3$.

%

A key observation in what follows, is that if $x,y$ are two vertices of a cell in $\L$,
then $||x-y||_\infty \leq 2$.
Thus, for any two subsets $A,B$ of edges of cells in $\L$ such that
the $L^\infty$-distance between $A$ and $B$ greater than $2$,
the configuration of tricolor edges in $A$
is independent of the configuration of tricolor edges in $B$.

For each $z \in \Z^3$, declare $z$ open if $C_r(2r z) \conn (C_{3r}(2r z))^c$.
If $\dist_{\Z^3}(z,z') \geq 5$ then the $L^{\infty}$-distance between
$C_{3r}(2r z)$ and $C_{3r}(2r z')$ is at least $r > 2$,
so vertices in $\Z^3$ of distance at least $5$ are independent.

Also, note that the event $C_R(x) \conn (C_{3R}(x))^c$ implies that
there exist vertices $z,z' \in \Z^3$ with
$R - r \leq || 2r z - x ||_\infty \leq R$ and $3R \leq || 2r z' - x ||_\infty \leq 3R+r$
such that $z$ and $z'$ are connected by an open path in $\Z^3$.
The number of possible choices for such $z,z'$ is polynomial in $R/r$,
and $\dist_{\Z^3}(z,z') \geq R/r$.
By Proposition \ref{prop:Peierls}, the probability of this is at most
$ c_1 \exp \sr{ - c_2 R/r}$, where $c_1,c_2$ are universal constants,
provided that the probability that any vertex $z$ is open is at most some fixed $\alpha <1$.

Thus, for some $\alpha<1$,  if $\sup_x \Pr_{f} [ C_r(x) \conn (C_{3r}(x))^c ] \leq \alpha$,
then for all $R$, and any $x$,
$$ \Pr_{f} [ C_R(x) \conn (C_{3R}(x))^c ] \leq c_1 \exp \sr{ - c_2 \tfrac{R}{r} } . $$
\end{proof}

%
%
%



The proof of Theorem \ref{thm::exponentialdecayofpathlength} is now straightforward.

\begin{proof}[Proof of Theorem \ref{thm::exponentialdecayofpathlength}]
Since $p$ is in the compact phase, we may choose $r>0$ large enough so that
for any $z \in \Z^3$ we have that
$\Pr_p [ C_r(2r z) \conn (C_{3r}(2r z))^c ] < \alpha$, with
$\alpha$ as in Proposition \ref{prop:Peierls}.

If the origin is on a tricolor path of length $L$,
then by subsequently removing annuli of the form $C_{3r}(2r z) \setminus C_r(2r z)$,
we may find a connected subset $0 \in S \subset \Z^3$,
so that $|S| \geq \frac{L}{K r^3}$ for some
constant $K>0$ and
such that for every $z \in S$, $C_r(2r z) \conn (C_{3r}(2r z) )^c$
(the corresponding annulus is crossed by a tricolor path).
That is, the event that the origin is on a tricolor path of length at least $L$
implies that in the induced dependent site percolation on $\Z^3$,
$0$ is in an open component of size at least $\frac{L}{K r^3}$.
Since we chose $r$ so that any site is open with probability at most $\alpha$,
by Proposition \ref{prop:Peierls} we have that this probability
is at most $c_1 \exp \sr{ - c_2 \tfrac{L}{K r^3} }$, which is exponentially decreasing in $L$.
\end{proof}



\subsection{Proof of Theorem \ref{thm::closedandnonempty}}

In this section, we will consider a prism obtained by starting with the triangular array of cells shown in Figure \ref{fig::triangularlayer} and coloring randomly in the manner described in the caption to Figure \ref{fig::randomtriangularlayer}.  We then repeat for translations of the triangular array in the orthogonal direction, as shown in Figure \ref{fig::triangularprism}.





To make this more formal, consider the sets
$$W_n : = \set{ x \in \L \ : \  x_1+x_2+x_3 = n, x_1 > 0, x_2 > 0, x_3 > 0 }.$$
Then $L_n:= W_{n-1} \cup W_n \cup W_{n+1}$ is a triangular array of the sort shown in
Figure \ref{fig::triangularlayer}.  If we then define $L_{k,n} = L_n + (k,k,k)$, then the prism in Figure \ref{fig::triangularprism} is a union of $L_{k,n}$ layers over a range of $k$ values.  Write $T_n = \bigcup_{k=-\infty}^\infty L_{k,n}$.



To every cell $v \in L_n$ we may associate a probability vector $\mbf{p}_{v,n} \in \mathcal T$,
by letting $\mbf{p}_{v,n} = (p_1,p_2,p_3)$ where $v = p_1 v_1 + p_2 v_2 + p_3 v_3$ and
$v_1,v_2,v_3$ are the corner vertices of $L_n$.
Also, to every $v \in L_{k,n}$ we may associate $\mbf{p}_{v,n} = \mbf{p}_{u,n}$ where
$v = u+ (k,k,k)$.

Let $\Pr_n$ be the law of the varying coloring of the cells in $T_n$, by coloring each cell $v$
independently using the probability vector $\mbf{p}_{v,n}$.

\begin{figure}[htbp]
\begin{center}
\includegraphics[width=3.8in]{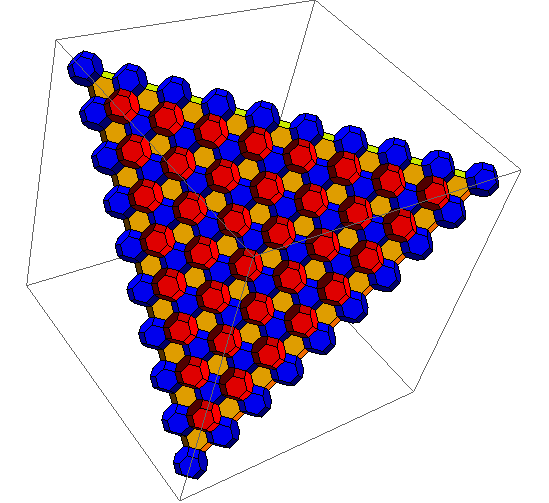}
\caption{ \label{fig::triangularlayer} 
One layer of a prism, which is a union of subsets $W_{n-1},W_n,W_{n+1}$.}
\end{center}
\end{figure}

\begin{figure}[htbp]
\begin{center}
\includegraphics[width=3.8in]{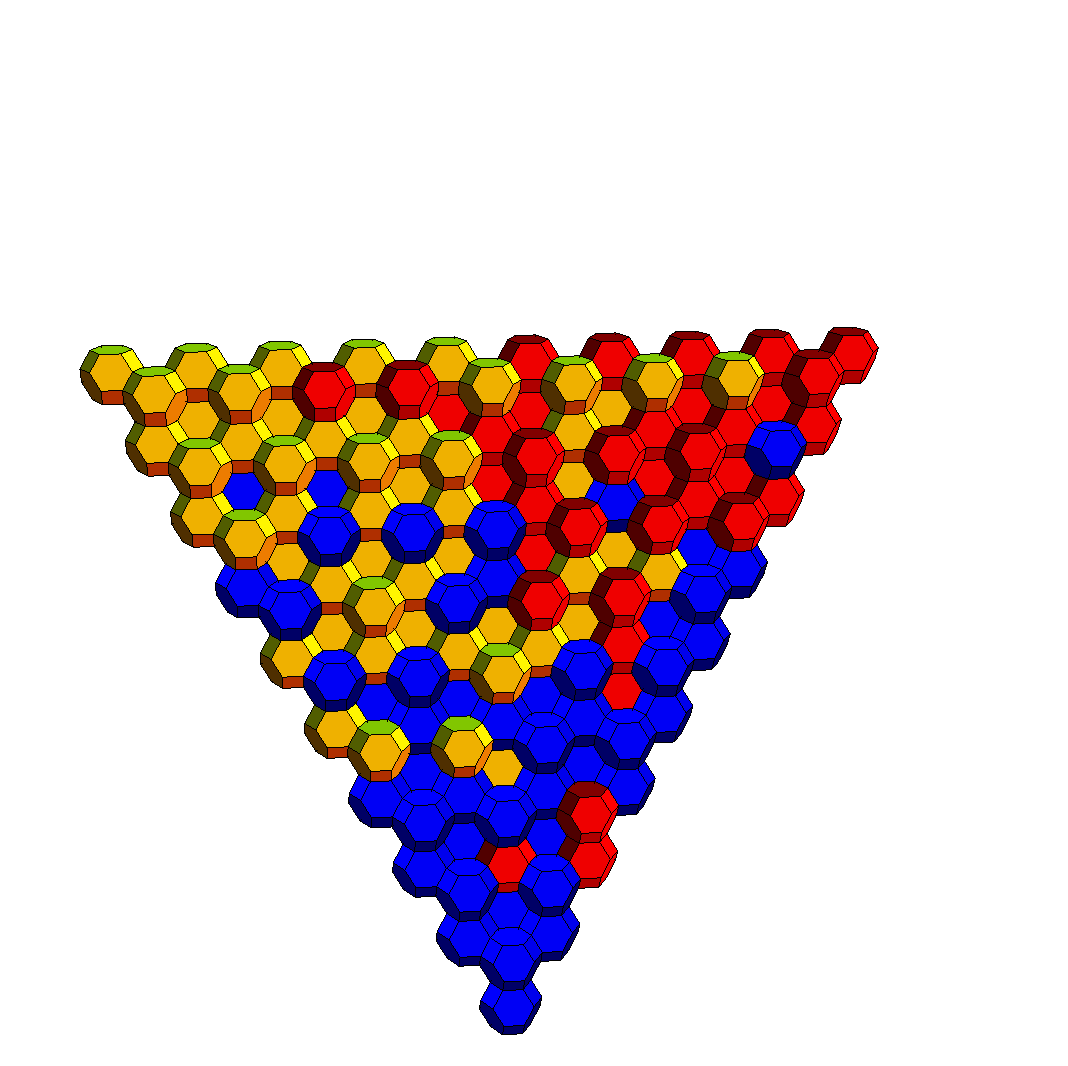}
\caption{ \label{fig::randomtriangularlayer} 
A depiction of the layer $L_{k,n}$ from Figure \ref{fig::triangularlayer} together with a random coloring.  Each cell is colored independently of all others, but the probability vector depends on the location of the cell.  Consider one of the three levels from Figure \ref{fig::triangularlayer}, and label the corner vertices $v_1, v_2, v_3$.  The each $v$ on that level is colored according to the probability vector $(p_1, p_2, p_3)$ for which $v = p_1 v_1 + p_2 v_2 + p_3 v_3$.  The same is done for the other two levels.
}
\end{center}
\end{figure}

\begin{figure}[htbp]
\begin{center}
\includegraphics[width=3.8in, trim = 0in 0in 0in 5in, clip = true]{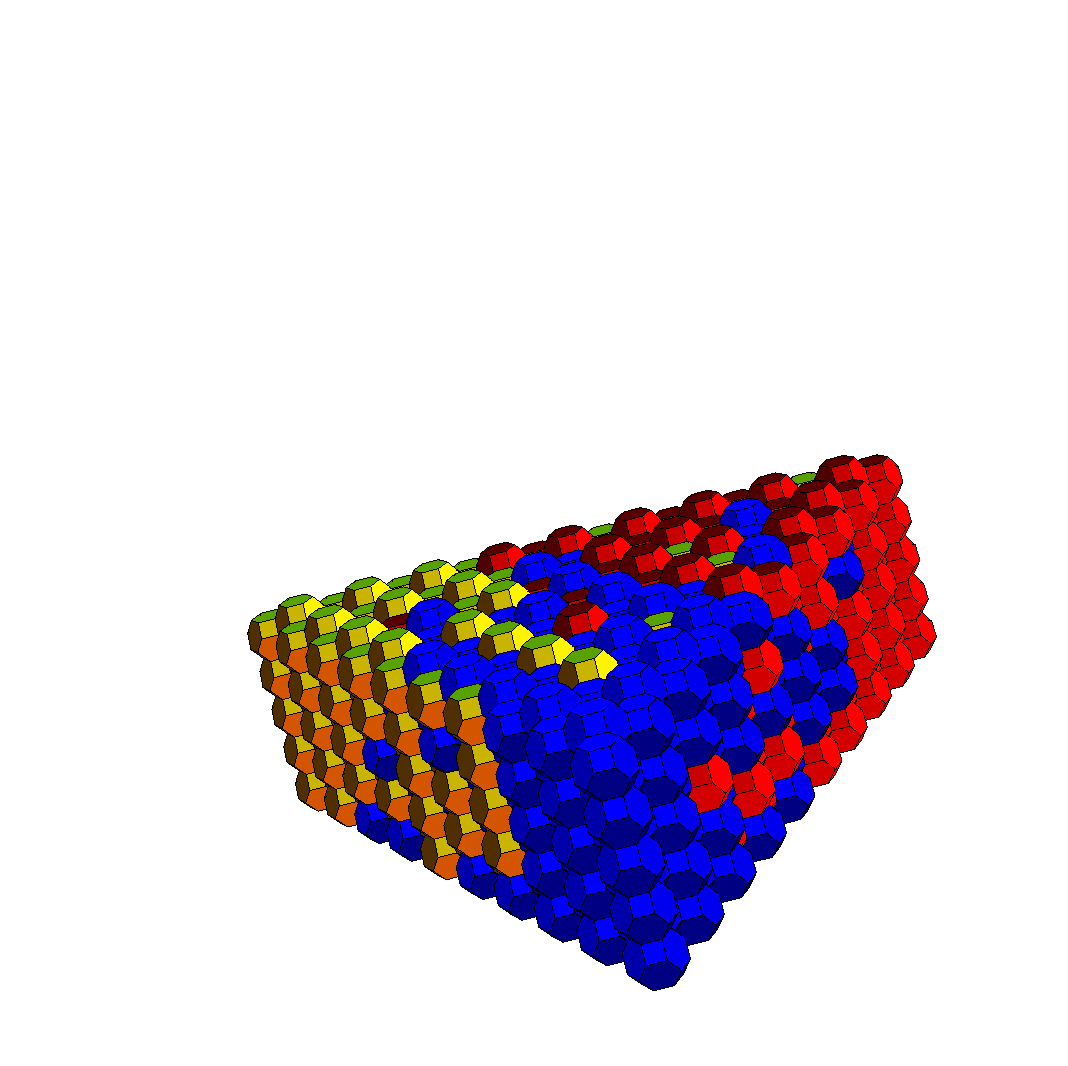}
\caption{ \label{fig::triangularprism} 
This prism-shaped figure is produced by stacking independent copies of the randomly colored triangles described in Figure \ref{fig::randomtriangularlayer}.  Observe that on each of three vertical sides, only two color possibilities are allowed, and on each of three vertical edges, only one color is allowed.  In particular, there are no tricolor edges entering or exiting the prism on any of these sides.  On the other hand, Proposition \ref{prop::flowthroughsurface} implies that there is a net flow of one unit through each layer of the prism, and hence there must be least one tricolor path that passes through the inside of the prism, from the upper triangular face to the lower triangular face.
}
\end{center}
\end{figure}

\begin{prop}
\label{prop:infinite paths in varying}
For any $n$ consider the varying tricolor percolation on $T_n$ with law $\Pr_n$.
Then, $\Pr_n$-a.s.\ there must exist an infinite tricolored path in $T_n$.
\end{prop}

\begin{proof} Apply Proposition \ref{prop::flowthroughsurface} to a horizontal surface whose boundary passes through the boundary cells of the triangular array of Figure \ref{fig::randomtriangularlayer} in clockwise order, and such that the surface itself is contained within the array of cells in Figure \ref{fig::randomtriangularlayer}.  Proposition \ref{prop::flowthroughsurface} implies that the net amount of flow through this surface must be $1$.  On the other hand, since there are no tricolor edges entering or exiting $T_n$, we deduce that if all tricolor loops in $T_n$ were finite, then the net amount of flow through any such surface would have to be zero.


\end{proof}

\begin{proof}[Proof of Theorem \ref{thm::closedandnonempty}]
For any $p$ in the compact phase
there exist $C,c > 0$ and $r = r(p)$ such that
such that for any $x \in \L$ and any $R>0$,
$$ \Pr_{p} [ C_R(x) \conn (C_{3R} (x))^c ] \leq C e^{-c R / r} . $$
For every $p$ in the compact phase let $R_0 = R_0(p)$
be large enough so that for all $R>R_0$ and all $x \in \L$,
$ \Pr_{p} [ C_R(x) \conn (C_{3R}(x))^c ] \leq \tfrac{\alpha}{4} , $
where $\alpha$ is the constant from Lemma \ref{lem:compact phase}.
By coupling with i.i.d.\ uniform random variables for each $x \in \L$,
we have that for all $q \in \mathcal{T}$
$$ \Pr_{q} [ C_R(x) \conn (C_{3R}(x) )^c ] \leq
\Pr_{p} [ C_R(x) \conn (C_{3R}(x))^c ] + 2 ||q-p ||_1 \cdot K R^3 , $$
for some universal constant $K>0$.
Thus for every $p$ in the compact phase
there exist $R(p), \eps(p) > 0$,
such that for any
$q \in \mathcal{T}$ with $|| q-p ||_1 < \eps (p)$ and all $r \geq R(p)$,
$$ \Pr_{q} [ C_r(x) \conn (C_{3r}(x))^c ] \leq \tfrac{\alpha}{2}   . $$
This proves that the compact phase is open, and so the extended phase is closed.

Under the assumption that the extended phase  is empty,
the compact phase is the whole of $\mathcal{T}$, and specifically compact.
Since $(\set{ q \in \mathcal{T}  \ : \  || q-p ||_1 < \eps (p) } )_{p \in \mathcal{T}}$
is an open cover of $\mathcal{T}$,
we may extract a finite sub-cover, say
$(\set{ q \in \mathcal{T}  \ : \  || q-p^{(j)} ||_1 < \eps(p^{(j)}) })_{j=1}^m$.
Taking $r : = \max_{1 \leq j \leq m}  R(p^{(j)})$,
we get that for all $p \in \mathcal{T}$ and all $x \in \L$,
$$ \Pr_{p} [ C_r(x) \conn (C_{3r}(x))^c ] \leq \tfrac{\alpha}{2} . $$

If $x \sim y$ then $||x-y||_1 \leq 2$ so
$|| \mbf{p}_{x,n} - \mbf{p}_{y,n} ||_1 \leq \tfrac{2}{n}$.
Using i.i.d.\ uniform random variables for each $x \in \L$, we may couple
$\Pr_n$ with $\Pr_{\mbf{p}_{x,n}}$ so that
the configuration on $B(x,r)$ is not identical with probability at most
$$ \sum_{y \in B(x,r) } \tfrac{4}{n} \cdot \dist(y,x) \leq \frac{K r^4}{n} , $$
for some universal constant $K>0$.
Thus, for any $x \in \L$,
if $K r^4 < \tfrac{\alpha}{2} n$ then
$$ \Pr_n [ C_r(x) \conn (C_{3r}(x) )^c ] \leq
\Pr_{\mbf{p}_{x,n}} [ C_r(x) \conn C_{3r}(x))^c ] + \tfrac{\alpha}{2}
\leq \alpha . $$
By Lemma \ref{lem:compact phase}, we conclude that
there exist constants $C, c = c(r) > 0$ such that for all $n> \tfrac2\alpha K r^4$,
for any $R>0$ and any $x \in T_n$,
$$ \Pr_n [ C_R(x) \conn (C_{3R}(x))^c ] \leq C e^{-c R} . $$

Finally, by Proposition \ref{prop:infinite paths in varying}
there a.s.\ exists an infinite tricolored path that goes through $L_{0,n}$.
That is, for any $R>0$ such that there a.s.\ exists $x \in L_{0,n}$ such that
$x$ is connected by a tricolored path to distance $3R$ in $T_n$.
This implies that for any $R>0$, there a.s.\ exists $x \in L_{0,n}$ such that
$C_R(x) \conn (C_{3R}(x))^c$.
Since $|L_{0,n}| \leq K n^2$ for some constant $K>0$,
\begin{align*}
1 & \leq \sum_{x \in L_{0,n} } \Pr_n [ C_R(x) \conn (C_{3R}(x) )^c ]
\leq K n^2 \cdot C e^{-c R} .
\end{align*}
Taking $R \to \infty$ gives a contradiction.
\end{proof}

\subsection{Proof of Theorem \ref{thm::orangegreenpurple}}

\begin{prop}
\label{prop:Lipschitz}
The green phase separation curve in Figure \ref{fig::phasediagram} is the graph of a Lipschitz function with Lipschitz norm at most $\sqrt{3}$.
\end{prop}

\begin{proof}
As usual, we use $1,2,3$ to indicate red, yellow and blue respectively. Let $\mathcal G$ be the set of $p = (p_1, p_2, p_3) \in \mathcal{T}$ for which the probability that the green face cluster containing a given face has more than $N$ faces does {\em not} decay exponentially fast with $N$.  We aim to show that there exists a Lipschitz curve, such as the one drawn in Figure \ref{fig::phasediagram}, such that $\mathcal G$ is the region bounded below that curve.

We first claim that if $p \in \mathcal{G}$ and $q \in \mathcal T$ with $q_2 \geq p_2$ and $q_3 \geq p_3$, then $q \in \mathcal G$.
This follows from the fact that we can couple $\Pr_p$ and $\Pr_q$ in such a way that if $(\sigma_p, \sigma_q)$ is sampled from the coupling then every cell that is yellow (resp.\ blue) in $\sigma_p$ is also yellow (resp.\ blue) in $\sigma_q$.
One explicit coupling is as follows.  Let $(U_x)_{x \in \L}$ be i.i.d.\ uniform-$[0,1]$ random variables.
Then for any probability vector $p$, let $\sigma_p(x)$ be yellow if $U_x < p_2$, blue if $1 - U_x < p_3$, and red otherwise.

We next observe that the above claim can be restated as the fact that for every $p \in \mathcal G$, the equilateral triangle {\em under} $p$ (i.e., the unique equilateral triangle with one vertex given by $p$ and an one edge given by a segment of the bottom edge in Figure \ref{fig::phasediagram}) belongs to $\mathcal G$ as well.  In particular, this means that the interior of $\mathcal G$ can be written as a union of interiors of equilateral triangles of this type (the triangles under $p$, as $p$ ranges over all of $\mathcal G$).  The interior of each such triangle is described by the set of points below a Lipschitz function (ignoring the points on the bottom edge itself), and the union is the set of points below the supremum of these functions.  The lemma then follows from the fact that the supremum of a family of Lipschitz functions (with given Lipschitz norm) is itself a Lipschitz function.  (Note that this argument does not tell us whether the points on the boundary curve are themselves members of $\mathcal G$.)
\end{proof}

\begin{thm}
\label{thm:pc and BC}
If $p \in \mathcal G$ then $p_3 \geq p_c \cdot \sr{ 1 + \sr{ \tfrac{p_1}{14} }^{14} }$,
where $p_c = p_c(\L)$ is the critical value for site percolation on $\L$.
\end{thm}

\begin{proof}
Fix $\eps>0$ small, fix some $\beta \in [0,1]$
and consider the following coloring.
The faces in $\L$ correspond to edges $x \sim y$ for $x,y \in \L$.
Let us consider directed edges (or faces); so set $E = \set{ (x,y) \ : \ x \sim y \in \L}$.
Let $(B(x,y))_{(x,y) \in E}$ be i.i.d.\ Bernoulli-$\eps$ random variables,
and let $(C_x)_{x \in \L}$ be i.i.d.\ Bernoulli-$\beta$ random variables, so that
all random variables are independent.
For every $x$ let $B_x = \sum_{y \sim x} B(x,y)$.
Note that $0 \leq B_x \leq 14$ because $14$ is the degree in $\L$.

For every cell $x \in \L$ set
$$ \sigma(x) =
\begin{cases}
1 & \textrm{ if } C_x = 0 , B_x > 0 \\
2 & \textrm{ if } C_x = 0 , B_x = 0 \\
3 & \textrm{ if } C_x = 1 .
\end{cases}
$$

Note that $\sigma$ has the law $\Pr_{p}$
for $p = ( (1-\beta) (1- (1-\eps)^{14} ) , (1-\beta) (1-\eps)^{14} , \beta )$.

Let $G = \set{ x \ : \ \sigma(x) = 3 , \exists \ y \sim x \ : \ B(y,x) = 0 }$;
that is, $G$ is the set of all cells colored $3$ (blue) with at least one incoming
face that has $B(y,x) = 0$.
Note that if $x \sim y$ such that $\sigma(x) = 3, \sigma(y) = 2$ then
$B(y,x) = 0$ so $x \in G$.
So any component of faces colored $2,3$ (green) must be on the boundary
of a component of cells in $G$.
That is, if there exists a green bi-colored component of size $N$ passing near the origin, then $G$ contains a component whose size is also of order $N$ (up to a constant factor).

Now, note that $x \in G$ if and only if $C_x = 1$ and $\sum_{y \sim x} B(y,x) < 14$.
Since all these events are independent, we have that $G$ is just site percolation
on $\L$ with parameter $\beta (1-\eps^{14})$.
Thus, if $\beta (1-\eps^{14}) < p_c(\L)$ then the size of the component of $G$ incident to a given face has exponentially decaying law, and so $p \not\in \mathcal G$.

Thus, using $(1- (1-\eps)^{14}) \leq 14 \eps$, if
$(14 \eps,1-\beta-14\eps, \beta) \in \mathcal G$ then also
$((1-\beta) (1-(1-\eps)^{14}) , (1-\beta)(1-\eps)^{14} , \beta ) \in \mathcal G$.
So $\beta \geq p_c (1 - \eps^{14} )^{-1} \geq
p_c \cdot (1+ \eps^{14})$.
\end{proof}

\begin{proof}[Proof of Theorem \ref{thm::orangegreenpurple}]
Theorem \ref{thm::orangegreenpurple} now follows
from Proposition \ref{prop:Lipschitz} and Theorem \ref{thm:pc and BC}.
\end{proof}

There are some simple but interesting consequences of the above results.

\begin{cor}
If $p=(p_1,p_2,p_3)$ is in the extended phase then $p_j \geq
p_c (1+ \sr{ \tfrac{p_c}{14} }^{14} )$ for all $j =1,2,3$, where $p_c = p_c(\L)$.
In particular, Corollary \ref{cor:extended separated} holds.
\end{cor}

\begin{cor} \label{cor::explicitcriticalpercolationbound}
If $p_c(\L)$ is the critical threshold for site percolation on $\L$, then $p_c(\L) < \tfrac13$.
\end{cor}
\begin{proof}
The extended phase is non-empty, so let $p \in \mathcal{T}$ be a probability vector in the extended phase.  There exists $j$ such that $p_j \leq \tfrac13$.
So $p_c < p_c (1+ (p_c/14)^{14} ) \leq p_j \leq \tfrac13$.  The statement that $p_c (1+(p_c/14)^{14}) \leq 1/3$ implies that $p_c \leq x$ where $x$ is the positive solution to
$x(1+(x/14)^{14}) = 1/3$.  This bound is explicit, but the $x$ described this way is only barely below $1/3$.  Numerically we find
$x \approx .333333333333333333333327$.
\end{proof}




\appendix

\section{Permutohedral Lattice in dimension $d$} \label{sec::highdpermutohedra}

In this section we present a short practical overview the classical tessellation of finite dimensional space by permutohedra --- sufficient to allow the reader to perform computer simulations of the type described in Section \ref{subsec::overview}.  A more thorough treatment may be found in e.g.\ \cite{sloane1999sphere, ziegler1995lecture}.    The {\bf permutohedron of order $d$} can be defined as the convex hull of the points in $\R^d$ defined by $(\sigma(1), \sigma(2), \ldots, \sigma(d))$, where $\sigma$ ranges over the $d!$ possible permutations of $\{1,2,\ldots, d \}$.
These vertices all belong to the hyperplane $\{v : v_1+v_2 \ldots + v_d = d(d+1)/2 \}$.  In fact, it is often convenient to center this construction at the origin, so we define the centered permutohedron $\perm_d$ to be the convex hull of the vectors $$\Biggl( \sigma(1) - \frac{d+1}{2}, \ldots, \sigma(d) - \frac{d+1}{2} \Biggr),$$
as in Figure \ref{fig::convexhulls}.  Then $\perm_d$ lies in the subspace of $\R^d$ orthogonal to the vector $(1,1,\ldots,1)$.  We denote this vector by $\mbf{1}$, and the orthogonal subspace by $\mbf{1}^{\perp}$.

\begin{figure}[htbp]
\begin{center}
\includegraphics[width=0.4\textwidth]{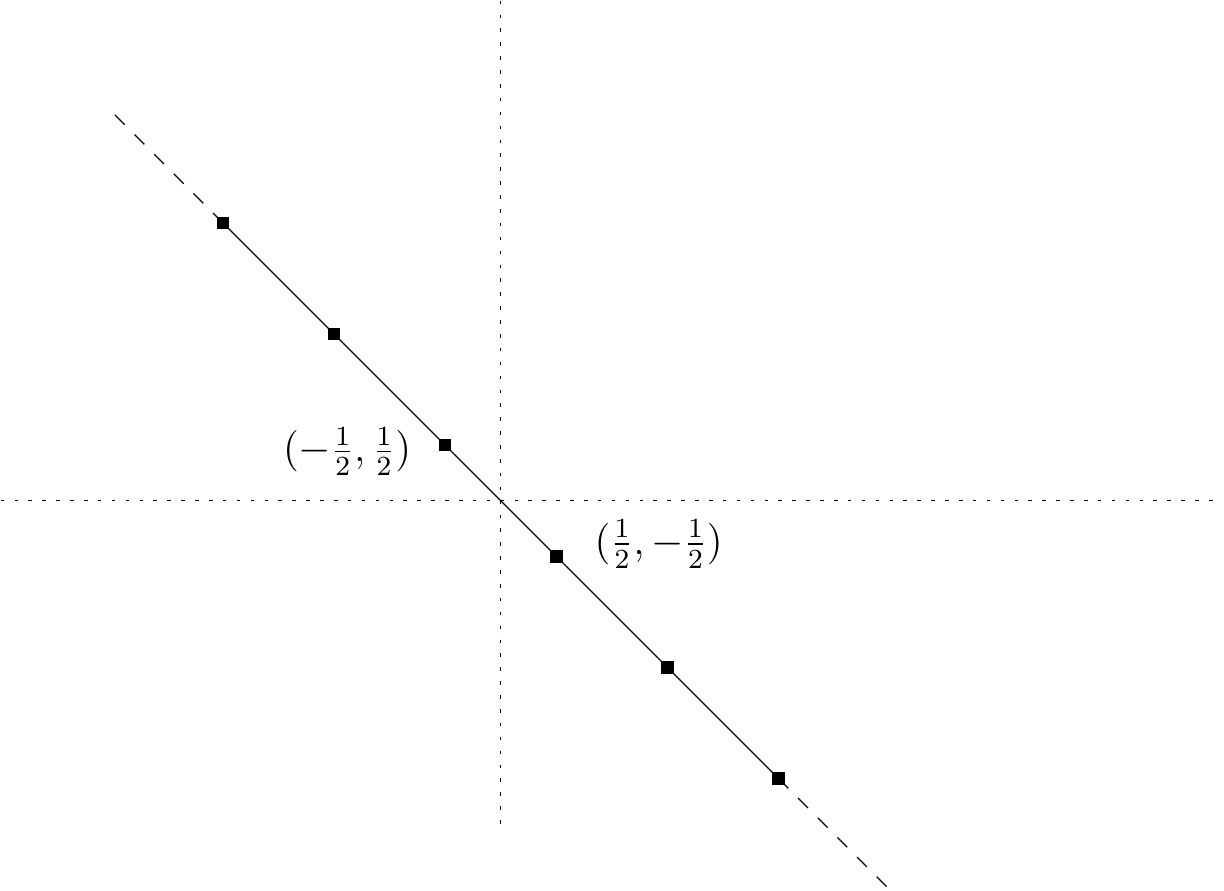} \hspace{.1in} \includegraphics[width=0.3\textwidth]{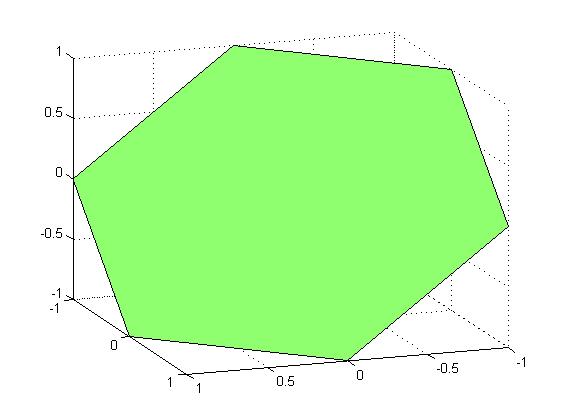}
\end{center}
\caption{ \label{fig::convexhulls} The convex hull of the set of vectors obtained by permuting the coordinates of $(-1/2,1/2)$ is a line segment embedded in $\R^2$.  Translates of this line segment tile the line of points orthogonal to $(1,1)$ in the obvious way.  The convex hull of the set of vectors obtained by permuting the coordinates of $(-1,0,1)$ is a hexagon embedded in $\R^3$.  Translates of this hexagon tile the plane orthogonal to $(1,1,1)$.}
\end{figure}


Let $\L_d$ be the image of the lattice $d \mathbb Z^d$ under the orthogonal projection map sending $\R^d$ to $\mbf{1}^{\perp}$.  In other words, $\L_d$ is the subset of $\mathbb Z^d$ consisting of vertices whose coordinates sum to zero and are all equal to each other modulo $d$.  For example, if $d=5$, then $(2,7,-3,-8,2) \in \L_d$.  The permutohedron tessellation of $\mbf{1}^{\perp} \equiv \R^{d-1}$ is simply the Voronoi tessellation of $\L^d$.

\begin{prop}
\label{prop:perm=voronoi cell}
In the space $\mbf{1}^{\perp}$,
the permutohedron $\perm_d$ is the Voronoi cell (or Dirichlet region)
of $0$ in the lattice $\L_d$.
For any $x \in \L_d$, the Voronoi cell of $x$ is
$\perm_d + x$.
That is,
$$ \perm_d + x = \set{ z \in \mbf{1}^{\perp} \ : \ \forall \ y \in \L_d \ ||z-x|| \leq ||z-y|| } . $$
Thus, these cells tessellate the space $\mbf{1}^{\perp}$.
\end{prop}


For a non-trivial subset $\emptyset \neq F \subsetneq \set{1,\ldots, d}$
let $v_F$ be the vector
$$ v_F (j) =
\begin{cases}
|F| - d & j \in F \\
|F| & j \not\in F .
\end{cases} $$
Note that, viewing $v_F$ as an ordered $d$-tuple, we have $v_F \in \mbf{1}^{\perp}$ and all entries of $v_F$
are in the same class modulo $d$.  In fact, the $v_F$ are precisely the non-zero elements of $\L_d$ whose coordinates all have absolute value less than $d$.
One may then observe that $\L_d$ is the lattice generated by
$\set{ v_F \ : \ \emptyset \neq F \subsetneq \set{1,\ldots, d} }$.  It turns out that $v,w \in \L_d$ correspond to adjacent cells in the permutohedron tessellation if and only if
$(v-w)$ belongs to this set.  Equivalently, $v$ and $w$
are adjacent if all coordinates of $v-w$ have absolute value less than $d$.

By construction, the vertices of the permutohedron $\perm_d$ are in one-to-one correspondence
with permutations.  Moreover, it is not hard to see that the vertices described by permutations $\sigma$ and $\tau$ lie on a common edge of $\perm_d$ if and only if
$\sigma = (j \ j+1) \tau$
for some $j \in \set{1,\ldots, d-1}$,
where $(j \ j+1)$ is the transposition of $j$ and $j+1$.
The graph with these edges is the Cayley graph of $S_d$ with respect to
the generating set of all transpositions of the form $(j \ j+1)$ where $
j \in \{ 1,\ldots, d-1 \}$.

Define $\L_d^*$ as the graph whose vertices are the translates of $S_d-\mbf{c}$
by elements of $\L_d$, and for two translates $x^* = x+\sigma - \mbf{c}$
and $y^* = y + \tau - \mbf{c}$, declare $x^* \sim y^*$ if
$x=y$ and $\sigma \sim \tau$ in $S_d$.  Every vertex of $\L_d^*$ is at the intersection of
a clique of $d$ cells in $\L_d$.  Moreover,
any edge in $\L_d^*$ is the intersection of a clique of $d-1$ cells
of $\L_d$.
Thus, there is a bijection between $d$-cliques in $\L_d$ and vertices
of $\L_d^*$, and a bijection between $(d-1)$-cliques of $\L_d$
and  edges of $\L_d^*$.

\bibliographystyle{plain}
\bibliography{tricolor}

\end{document}